\documentclass[12pt]{amsart}

\usepackage{amsmath,amssymb,amsfonts,amsthm,latexsym,graphicx,multirow,hyperref}
\usepackage[all]{xy}

\newcommand\A{\mathrm{A}} \newcommand\AGL{\mathrm{AGL}} \newcommand\Alt{\mathrm{Alt}} \newcommand\Aut{\mathrm{Aut}}
\newcommand\Cay{\mathrm{Cay}} \newcommand\Cos{\mathrm{Cos}}
\newcommand\D{\mathrm{D}}
\newcommand\E{\mathrm{E}}
\newcommand\F{\mathrm{F}} \newcommand\Fix{\mathrm{Fix}}
\newcommand\G{\mathrm{G}} \newcommand\Ga{\Gamma} \newcommand\GL{\mathrm{GL}}
\newcommand\Nor{\mathbf{N}}
\newcommand\PGL{\mathrm{PGL}} \newcommand\PGaL{\mathrm{P\Gamma L}} \newcommand\PSL{\mathrm{PSL}} \newcommand\PSp{\mathrm{PSp}} \newcommand\PSU{\mathrm{PSU}}
\newcommand\Soc{\mathrm{Soc}} \newcommand\Sy{\mathrm{S}} \newcommand\Sym{\mathrm{Sym}} \newcommand\Sz{\mathrm{Sz}}
\newcommand\Z{\mathbf{Z}} \newcommand\ZZ{\mathrm{C}}

\newtheorem{theorem}{Theorem}[section]
\newtheorem{lemma}[theorem]{Lemma}

\newtheorem{question}[theorem]{Question}

\theoremstyle{definition}

\newtheorem{construction}[theorem]{Construction}

\newtheorem*{remark}{Remark}

\begin{document}

\title[Nonnormal cubic Cayley graphs]{An infinite family of cubic nonnormal Cayley graphs on nonabelian simple groups}

\author[Chen]{Jiyong Chen}
\address{School of Mathematical Sciences\\Peking University\\ Beijing, 100871\\ P. R. China}
\email{cjy1988@pku.edu.cn}

\author[Xia]{Binzhou Xia}
\address{School of Mathematics and Statistics\\The University of Melbourne\\Parkville, VIC 3010\\Australia}
\email{binzhoux@unimelb.edu.au}

\author[Zhou]{Jin-Xin Zhou}
\address{Department of Mathematics\\Beijing Jiaotong University\\ Beijing, 100044\\ P. R. China}
\email{jxzhou@bjtu.edu.cn}

\maketitle

\begin{abstract}
We construct a connected cubic nonnormal Cayley graph on $\mathrm{A}_{2^m-1}$ for each integer $m\geqslant4$ and determine its full automorphism group. This is the first infinite family of connected cubic nonnormal Cayley graphs on nonabelian simple groups.


\end{abstract}

\section{Introduction}

In this paper all graphs considered are finite, simple and undirected. Given a group $G$ and an inverse-closed subset $S$ of $G\setminus\{1\}$, the \emph{Cayley graph} $\Cay(G,S)$ on $G$ with respect to $S$ is the graph with vertex set $G$ such that two vertices $x$ and $y$ are adjacent if and only if $yx^{-1}\in S$. Let $\widehat{G}$ be the right regular representation of $G$. It is easy to see that $\widehat{G}$ is a subgroup of $\Aut(\Cay(G,S))$. Moreover, it was shown by Godsil~\cite{Godsil1981} that the normalizer of $\widehat{G}$ in $\Aut(\Cay(G,S))$ is $\widehat{G}\rtimes\Aut(G,S)$, where $\Aut(G,S)$ is the group of automorphisms of $G$ fixing $S$ setwise. In particular, $\Aut(\Cay(G,S))=\widehat{G}\rtimes\Aut(G,S)$ if and only if $\widehat{G}$ is normal in $\Aut(\Cay(G,S))$. Viewing this, Xu in~\cite{Xu1998} introduced the concept of normal Cayley graphs: a Cayley graph $\Cay(G,S)$ is said to be \emph{normal} if $\widehat{G}$ is normal in $\Aut(\Cay(G,S))$. The study of normality of a Cayley graph plays an important role in the study of its automorphism group because once a Cayley graph $\Cay(G,S)$ is known to be normal, to determine its full automorphism group one only needs to determine the group $\Aut(G,S)$, which is usually much easier. For a survey paper on normality of Cayley graphs we refer the reader to~\cite{FLX2008}.

The normality of cubic Cayley graphs on nonabelian simple groups has received considerable attention. It was proved in~\cite{Praeger1999} that a connected cubic Cayley graph $\Cay(G,S)$ with $G$ nonabelian simple is normal if $\widehat{G}\rtimes\Aut(G,S)$ is transitive on the edge set of $\Cay(G,S)$. A graph is said to be \emph{arc-transitive} if its automorphism group acts transitively on the set of arcs. In~\cite{XFWX2005,XFWX2007} it was proved that the only connected arc-transitive cubic nonnormal Cayley graphs on nonabelian simple groups are two Cayley graphs on $\A_{47}$ up to isomorphism, and their full automorphism groups are both isomorphic to $\A_{48}$. On the other hand, examples of connected cubic nonnormal Cayley graphs on nonabelian simple groups are very rare. Since the connected arc-transitive cubic nonnormal Cayley graphs on nonabelian simple groups are only the above mentioned two graphs on $\A_{47}$, we can concentrate on the non-arc-transitive case. In this context, one has the following theorem combining~\cite[Theorem~1.1]{FLWX2002} and~\cite[Theorem~1.2]{ZF2010}.

\begin{theorem}\label{thm2}
\emph{(\cite{FLWX2002,ZF2010})} Let $\Cay(G,S)$ be a connected cubic nonnormal Cayley graph on a nonabelian simple group $G$. If $\Cay(G,S)$ is not arc-transitive, then one of the following holds:
\begin{itemize}
\item[(a)] $G=\A_{2^m-1}$ with $m\geqslant3$;
\item[(b)] $G$ is a simple group of Lie type of even characteristic except $\PSL_2(2^e)$, $\PSL_3(2^e)$, $\PSU_3(2^e)$, $\PSp_4(2^e)$, $\E_8(2^e)$, $\F_4(2^e)$, $\,^2\F_4(2^e)'$, $\G_2(2^e)$ and $\Sz(2^e)$.
\end{itemize}
\end{theorem}

Until recently, connected cubic nonnormal Cayley graphs on the groups listed in Theorem~\ref{thm2} were only found for $\A_{15}$ and $\A_{31}$~\cite{LLW2013}. In 2008, Feng, Lu and Xu asked the following question in their survey paper~\cite{FLX2008} on normality of Cayley graphs.

\begin{question}\label{qes1}
\emph{(\cite[Problem~5.9]{FLX2008})} Are there infinitely many connected nonnormal Cayley graphs of valency $3$ or $4$ on nonabelian simple groups?
\end{question}

Question~\ref{qes1} in the valency $4$ case has been answered by Wang and Feng~\cite{WF2012} in the affirmative. In this paper, we answer the question in the remaining case. Our main result is Theorem~\ref{thm1}, which gives a positive answer to Question~\ref{qes1}.

\begin{theorem}\label{thm1}
For each integer $m\geqslant4$, there exists a graph $\Ga_m$ satisfying:
\begin{itemize}
\item[(a)] $\Ga_m$ is a connected cubic nonnormal Cayley graph on $\A_{2^m-1}$;
\item[(b)] $\Ga_m\cong\Cay(\A_{2^m-1},S)$ for some set $S$ of three involutions in $\A_{2^m-1}$ such that $\Aut(\A_{2^m-1},S)=1$;
\item[(c)] $\Aut(\Ga_m)\cong\A_{2^m}$.
\end{itemize}
\end{theorem}

We call a Cayley graph $\Cay(G,S)$ a \emph{graphical regular representation} (\emph{GRR} for short) of $G$ if $\Aut(\Cay(G,S))=\widehat{G}$. Note that a GRR is necessarily a normal Cayley graph, and a necessary condition for $\Cay(G,S)$ to be a GRR is that $\Aut(G,S)=1$. In many circumstances it is shown that this condition is also sufficient, see for example \cite{FLWX2002,Godsil1981,Godsil1983}. More generally, a problem is posed in~\cite{FLWX2002} to determine the groups $G$ such that a Cayley graph $\Cay(G,S)$ on $G$ is a GRR of $G$ if and only if $\Aut(G,S)=1$. We remark that our graph $\Ga_m$ in Theorem~\ref{thm1} as a Cayley graph on $G:=\A_{2^m-1}$ is not only nonnormal (and hence not a GRR) but also satisfies the condition $\Aut(G,S)=1$. It is also worth remarking that, although the graph $\Gamma_m$ is not arc-transitive, it has local action $\ZZ_2$ so that it corresponds to a tetravalent arc-transitive graph in the standard way described in~\cite[Section~4.1]{PSV2013}.

The paper is organized as follows. We shall first give the construction of $\Ga_m$ for Theorem~\ref{thm1} in Section~2. Then the entirety of section~3 will be devoted to proving the connectivity of $\Ga_m$. Finally in Section~4 we prove the remaining properties of $\Ga_m$ described in Theorem~\ref{thm1}, thus completing the proof of the theorem.

\section{Construction of $\Ga_m$}

We first introduce some notation that is fixed throughout this paper. Let $m\geqslant4$ be an integer,
\[
H=\langle a,b\mid a^4=b^2=(ab)^2=1\rangle\times\langle c_1\rangle\times\langle c_2\rangle\times\dots\times\langle c_{m-3}\rangle,
\]
where $c_1,c_2,\dots,c_{m-3}$ are involutions,
\[
K=\langle a^2,b,c_1,c_2,\dots,c_{m-3}\rangle
=\langle a^2\rangle\times\langle b\rangle\times\langle c_1\rangle\times\langle c_2\rangle\times\dots\times\langle c_{m-3}\rangle
\]
and $h=a\prod_{i=0}^{\lceil(m-5)/2\rceil}c_{2i+1}$. Clearly, $H$ is the direct product of a dihedral group $\D_8$ of order $8$ and an elementary abelian $2$-group of rank $m-3$, so that $|H|=2^m$. For the sake of convenience, put $c_{-1}=c_0=1$. Define $x\in\Aut(H)$ by letting
\[
a^x=a^{-1},\quad b^x=ab,\quad c_{2i+1}^x=c_{2i+1}\quad\text{and}\quad c_{2i+2}^x=a^2c_{2i+1}c_{2i+2}
\]
for $0\leqslant i\leqslant\lfloor(m-5)/2\rfloor$ and letting $c_{m-3}^x=a^2c_{m-3}$ in addition if $m$ is even. Define $\tau\in\Aut(K)$ by letting
\[
(a^2)^\tau=b,\quad b^\tau=a^2,\quad c_{2i+1}^\tau=c_{2i-1}c_{2i}c_{2i+2}\quad\text{and}\quad c_{2i+2}^\tau=c_{2i-1}c_{2i}c_{2i+1}
\]
for $0\leqslant i\leqslant\lfloor(m-5)/2\rfloor$ and letting $c_{m-3}^\tau=c_{m-3}$ in addition if $m$ is even. Note that $x$ and $\tau$ are indeed automorphisms of $H$ and $K$ because the images of generators under $x$ and $\tau$ satisfy the defining relations for $H$ and $K$, respectively. Denote the right regular representation of $H$ by $R$. Let $y$ be the permutation of $H$ such that $k^y=k^\tau$ and
\[
(hk)^y=
\begin{cases}
hk^\tau\quad&\text{if $m$ is odd,}\\
hk^\tau c_{m-3}\quad&\text{if $m$ is even}
\end{cases}
\]
for $k\in K$. Let
\[
z=
\begin{cases}
R(h)yR(h^{-1})\quad&\text{if $m$ is odd,}\\
R(h)yR(h^{-1}c_{m-3})\quad&\text{if $m$ is even.}
\end{cases}
\]
We will see that the three permutations $x$, $y$ and $z$ of $H$ are all involutions in $\Alt(H)$.

\begin{lemma}
$x$, $y$ and $z$ are all involutions.
\end{lemma}

\begin{proof}
It is evident that none of $x$, $y$ and $z$ is trivial. Since $x^2$ fixes each of the generators $a,b,c_1,\dots,c_{m-3}$ of $H$, we have $x^2=1$. Similarly, $\tau^2=1$. Let $g$ be an arbitrary element of $K$. Then $g^{y^2}=(g^\tau)^y=(g^\tau)^\tau=g^{\tau^2}=g$. If $m$ is odd, then $(hg)^{y^2}=(hg^\tau)^y=h(g^\tau)^\tau=hg^{\tau^2}=hg$ and so $y^2=1$, which in turn implies that $z^2=R(h)y^2R(h^{-1})=1$. Now assume that $m$ is even. Then
\[
(hg)^{y^2}=(hg^\tau c_{m-3})^y=h(g^\tau c_{m-3})^\tau c_{m-3}=hg^{\tau^2}c_{m-3}^\tau c_{m-3}=hg^{\tau^2}=hg,
\]
whence $y^2=1$. Moreover,
\begin{align*}
g^{z^2}&=g^{R(h)yR(c_{m-3})yR(h^{-1}c_{m-3})}\\
&=(hh^{-1}gh)^{yR(c_{m-3})yR(h^{-1}c_{m-3})}\\
&=\left(h(h^{-1}gh)^\tau\right)^{yR(h^{-1}c_{m-3})}\\
&=h\left((h^{-1}gh)^\tau \right)^\tau c_{m-3}h^{-1}c_{m-3}\\
&=h(h^{-1}gh)h^{-1}\\
&=g
\end{align*}
and
\begin{align*}
(hg)^{z^2}&=(hg)^{R(h)yR(c_{m-3})yR(h^{-1}c_{m-3})}\\
&=(hgh)^{yR(c_{m-3})yR(h^{-1}c_{m-3})}\\
&=\left((hgh)^\tau\right)^{R(c_{m-3})yR(h^{-1}c_{m-3})}\\
&=\left((hgh)^\tau c_{m-3}\right)^{yR(h^{-1}c_{m-3})}\\
&=\left((hgh)^\tau c_{m-3}\right)^\tau h^{-1}c_{m-3}\\
&=hghc_{m-3}h^{-1}c_{m-3}\\
&=hg.
\end{align*}
Thus $z^2=1$, completing the proof.
\end{proof}

\begin{lemma}\label{lem13}
$\Aut(H)\leqslant\Alt(H)$.
\end{lemma}

\begin{proof}
The conclusion for $m=4$ is easy to verify. Thus we assume $m\geqslant5$ in the following. Since the center $\Z(H)=\langle a^2,c_1,c_2,\dots,c_{m-3}\rangle$ is a characteristic subgroup of $H$, each automorphism $\sigma\in\Aut(H)$ induces an automorphism of $H/\Z(H)=\{\Z(H),a\Z(H),b\Z(H),ab\Z(H)\}$. More precisely, there is a homomorphism $\varphi$ from $\Aut(H)$ to $\Aut(H/\Z(H))$ such that $\varphi(\sigma)$ maps $g\Z(H)$ to $g^\sigma\Z(H)$ for all $\sigma\in\Aut(H)$ and $g\in H$.

Take an arbitrary $\sigma\in\Aut(H)$. Note that $a\Z(H)$ contains elements of order $4$ while the elements in $b\Z(H)$ and $ab\Z(H)$ are all involutions. We see that if $\varphi(\sigma)\neq1$ then $\varphi(\sigma)$ must fix the elements $\Z(H)$ and $a\Z(H)$ and swap $b\Z(H)$ and $ab\Z(H)$ in $H/\Z(H)$. Consequently, $\varphi(\sigma)\in\langle\varphi(x)\rangle$, and so $\sigma\in w\langle x\rangle$ for some $w\in\ker(\varphi)$. Since $w$ stabilizes $\Z(H)$, $a\Z(H)$ and $b\Z(H)$, we have
\[
a^w=a^{2\lambda+1}\prod_{j=1}^{m-3}c_j^{\lambda_j},\quad b^w=a^{2\mu}b\prod_{j=1}^{m-3}c_j^{\mu_j}\quad\text{and}\quad c_i^w=a^{2k_i}\prod_{j=1}^{m-3}c_j^{k_{i,j}}
\]
for each $i$ with $1\leqslant i\leqslant m-3$, where $\lambda$, $\mu$, $k_i$, $\lambda_j$ , $\mu_j$ and $k_{i,j}$ are all in $\{0,1\}$. Let $w_1$, $w_2$ and $w_3$ be automorphisms of $H$ such that
\[
\begin{array}{lll}
a^{w_1}=a^{2\lambda+1},&b^{w_1}=a^{2\mu}b,&c_i^{w_1}=a^{2k_i}c_i,\\
a^{w_2}=a,&b^{w_2}=b,&c_i^{w_2}=\prod_{j=1}^{m-3}c_j^{k_{i,j}},\\
a^{w_3}=a\prod_{j=1}^{m-3}c_j^{\lambda_j},&b^{w_3}=b\prod_{j=1}^{m-3}c_j^{\mu_j},&c_i^{w_3}=c_i.\\
\end{array}
\]
Then $w_1$ and $w_3$ are involutions, and $w=w_1w_2w_3$.

For each $\rho\in\Aut(H)$, the set of fixed points of $\rho$ is a subgroup of $H$ and thereby has size $2^\ell$ for some integer $\ell$ such that $0\leqslant\ell\leqslant m$. Thus, each involution of $\Aut(H)$ with at least four fixed points lies in $\Alt(H)$ as it is a product of $(|H|-2^\ell)/2=2^{m-1}-2^{\ell-1}$ transpositions for some integer $\ell$ such that $2\leqslant\ell\leqslant m-1$. Since $x$ and $w_3$ fix every point in $\langle a^2,c_1\rangle$ and $w_1$ fixes every point in $\langle a^{\lambda+1},a^\mu b^\lambda\rangle$, we then conclude that $x,w_3,w_1\in\Alt(H)$. Moreover, since $(gv)^{w_2}=gv^{w_2}$ for all $g\in\langle a,b\rangle$ and $v\in\langle c_1,\dots,c_{m-3}\rangle$, the number of transpositions of $w_2$ is divisible by $|\langle a,b\rangle|=8$. In particular, $w_2\in\Alt(H)$. Now $w_1$, $w_2$, $w_3$ and $x$ are all in $\Alt(H)$. It follows that $\sigma\in\Alt(H)$ due to $\sigma\in w\langle x\rangle=w_1w_2w_3\langle x\rangle$. This shows that $\Aut(H)\leqslant\Alt(H)$.
\end{proof}

The next lemma says that $x$ and $y$ as well as the elements of $R(H)$ are all even permutations on $H$. Note that this also implies $z\in\Alt(H)$ since $z\in\langle y,R(H)\rangle$.

\begin{lemma}\label{lem7}
$\langle x,y,R(H)\rangle\leqslant\Alt(H)$.
\end{lemma}

\begin{proof}
Lemma~\ref{lem13} already indicates $x\in\Alt(H)$. Let $\sigma$ be the map from $K$ to $hK$ sending $g$ to $hg$ for all $g\in K$, and $t$ be the permutation on $H$ such that $g^t=g$ and $(hg)^t=hgc_{m-3}^{m-1}$ for all $g\in K$. Then $t$ is the identity permutation if $m$ is odd, and is a product of $|K|/2$ transpositions if $m$ is even. In particular, $t\in\Alt(H)$. From the definition of $y$ one sees that the following diagram commutes.
\[
\xymatrix{
K\ar[r]^{yt}\ar[d]^{\sigma}&K\ar[d]^{\sigma}\\
hK\ar[r]^{yt}&hK
}
\]
Hence $(yt)|_{hK}$ has the same cycle structure as $(yt)|_K$, and so $yt\in\Alt(H)$. This in turn gives $y\in\Alt(H)$. Finally, as $H$ is a $2$-group and not cyclic, we have $R(H)\leqslant\Alt(H)$. Consequently, $\langle x,y,R(H)\rangle\leqslant\Alt(H)$.
\end{proof}

Recall the standard construction of the \emph{coset graph} $\Cos(G,H,HSH)$ given a group $G$ with a subgroup $H$ and a subset $S$ such that $S\cap H=\emptyset$ and $HSH$ is inverse-closed. Such a graph has vertex set $[G{:}H]$, the set of right cosets of $H$ in $G$, and edge set $\{\{Hg,Hsg\}\mid g\in G,\ s\in HSH\}$. It is easy to see that $\Cos(G,H,HSH)$ has valency $|HSH|/|H|$, and $G$ acts by right multiplication on $[G{:}H]$ as a group of automorphisms of $\Cos(G,H,HSH)$. Moreover, $\Cos(G,H,HSH)$ is connected if and only if $\langle S,H\rangle=G$.

Now we are in the position to construct the graph $\Ga_m$ for Theorem~\ref{thm1}.

\begin{construction}
For each integer $m\geqslant4$, let
\[
\Ga_m=\Cos(\Alt(H),R(H),R(H)\{x,y\}R(H))
\]
with $H$, $x$ and $y$ defined at the beginning of this section.
\end{construction}

\section{Connectivity of $\Ga_m$}

The aim of this section it to prove that $\Ga_m$ is connected. According to the construction of $\Ga_m$, it suffices to prove $\langle x,y,R(H)\rangle=\Alt(H)$, and we will achieve this by dealing with the cases $m$ is odd and $m$ is even separately. For a group $G$, denote the set $G\setminus\{1\}$ by $G^*$. For a permutation $\sigma$ of a set $\Omega$ and $\alpha,\beta\in\Omega$, we write $\alpha\xrightarrow\sigma\beta$ if $\alpha^\sigma=\beta$.

\subsection{Technical lemms}

We first establish two technical lemmas that will be needed later in this section.

\begin{lemma}\label{lem3}
Let $\ell\geqslant2$ be an even integer, and $V=\langle e_1\rangle\times\langle e_2\rangle\times\dots\times\langle e_\ell\rangle$ be a group with involutions $e_1,e_2,\dots,e_\ell$. Denote the right regular representation of $V$ by $r$. Let $\omega=r(e_{\ell-1}e_\ell)$, $e_{-1}=e_0=1$, and $\chi$ and $\psi$ be automorphisms of $V$ such that $e_{2i+1}^\chi=e_{2i+1}$, $e_{2i+2}^\chi=e_{2i+1}e_{2i+2}$, $e_{2i+1}^\psi=e_{2i-1}e_{2i}e_{2i+2}$ and $e_{2i+2}^\psi=e_{2i-1}e_{2i}e_{2i+1}$ for each $i$ with $0\leqslant i\leqslant(\ell-2)/2$. Then $\langle\chi,\psi,\omega\rangle$ is a transitive subgroup of $\Sym(V)$.
\end{lemma}

\begin{proof}
Note that, viewing $V$ as a vector space over $\mathbb{F}_2$, the vectors $e_1^\chi,e_2^\chi,\dots,e_\ell^\chi$ form a basis of $V$. Thus the automorphism $\chi$ of $V$ is well-defined. Similarly, $\psi$ is well-defined. Write $N=\langle\chi,\psi,\omega\rangle$. Since $\chi$ is an automorphism of $V$, we have
\[
r(e_\ell)=r((e_{\ell-1}e_\ell)^\chi)=\chi^{-1}r(e_{\ell-1}e_\ell)\chi=\chi^{-1}\omega\chi\in N
\]
and so $r(e_{\ell-1})=r(e_{\ell-1}e_\ell)r(e_\ell)=\omega r(e_\ell)\in N$. Suppose there exists a nonnegative integer $i\leqslant(\ell-2)/2$ such that $r(e_{\ell-2i+1}),r(e_{\ell-2i+2}),\dots,r(e_{\ell-1}),r(e_\ell)$ are all in $N$. Note that
\begin{align*}
r(e_{\ell-2i-1}e_{\ell-2i})&=r(e_{\ell-2i-1}e_{\ell-2i}e_{\ell-2i+1})r(e_{\ell-2i+1})\\
&=r(e_{\ell-2i+2}^\psi)r(e_{\ell-2i+1})=\psi^{-1}r(e_{\ell-2i+2})\psi r(e_{\ell-2i+1})
\end{align*}
since $\psi$ is an automorphism of $V$. We thereby deduce that $r(e_{\ell-2i-1}e_{\ell-2i})\in N$. It follows that  $r(e_{\ell-2i})=r((e_{\ell-2i-1}e_{\ell-2i})^\chi)=\chi^{-1}r(e_{\ell-2i-1}e_{\ell-2i})\chi\in N$ and thus $r(e_{\ell-2i-1})=r(e_{\ell-2i-1}e_{\ell-2i})r(e_{\ell-2i})\in N$. Then by induction one concludes that $r(e_1),r(e_2),\dots,r(e_{\ell-1}),r(e_\ell)$ are all in $N$. Consequently, $r(V)\leqslant N$ and so $N$ is transitive on $V$.
\end{proof}

The following lemma is a consequence of the classification of doubly transitive permutation groups (see for example~\cite{Cameron1999}).

\begin{lemma}\label{lem2}
Suppose that $G$ is a doubly transitive permutation group on $2^m$ points. Then one of the following holds:
\begin{itemize}
\item[(i)] $G\leqslant\AGL_m(2)$;
\item[(ii)] $2^m-1=q$ for some prime power $q$ and $\PSL_2(q)\leqslant G\leqslant\PGaL_2(q)$;
\item[(iii)] $\A_{2^m}\leqslant G\leqslant\Sy_{2^m}$.
\end{itemize}
\end{lemma}

\begin{remark}
In fact, the prime power $q$ in case~(ii) of Lemma~\ref{lem2} is necessarily a prime by Mih\v{a}ilescu's theorem~\cite{Mihailescu2004}. In particular, $m$ must be odd in case~(ii) of Lemma~\ref{lem2}.
\end{remark}

\subsection{Odd $m$}

Throughout this subsection, let $m$ be odd, and
\[
U=\left\{\prod_{i=1}^{m-3}c_i^{k_i}\mid\sum_{j=1}^{(m-3)/2}k_{2j}\equiv0\pmod{2}\right\}.
\]
Note that $\{U,Uc_{m-3}\}$ forms a partition of $\langle c_1,c_2,\dots,c_{m-3}\rangle$, and $x$ stabilizes $U$ setwise. For each $u\in U$ we have
\begin{equation}\label{eq1}
u\xrightarrow{x}u^x\xrightarrow{y}u^{xy}\xrightarrow{z}u^x\xrightarrow{x}u\xrightarrow{y}u^y\xrightarrow{z}u,
\end{equation}
\begin{equation}\label{eq2}
au\xrightarrow{x}a^{-1}u^x\xrightarrow{y}abu^{xy}c_{m-4}c_{m-3}\xrightarrow{z}au^x
\xrightarrow{x}a^{-1}u\xrightarrow{y}abu^yc_{m-4}c_{m-3}\xrightarrow{z}au,
\end{equation}
\begin{align}\label{eq3}
&uc_{m-3}\xrightarrow{x}a^2u^xc_{m-4}c_{m-3}\xrightarrow{y}bu^{xy}c_{m-4}c_{m-3}\xrightarrow{z}bu^xc_{m-4}c_{m-3}\\
\nonumber
&bu^xc_{m-4}c_{m-3}\xrightarrow{x}a^{-1}buc_{m-3}\xrightarrow{y}a^{-1}bu^yc_{m-6}c_{m-5}c_{m-3}\xrightarrow{z}a^{-1}buc_{m-3}\\
\nonumber
&a^{-1}buc_{m-3}\xrightarrow{x}bu^xc_{m-4}c_{m-3}\xrightarrow{y}a^2u^{xy}c_{m-4}c_{m-3}\xrightarrow{z}a^2bu^xc_{m-4}c_{m-3}\\
\nonumber
&a^2bu^xc_{m-4}c_{m-3}\xrightarrow{x}abuc_{m-3}\xrightarrow{y}a^{-1}u^yc_{m-6}c_{m-5}c_{m-4}\xrightarrow{z}a^{-1}uc_{m-3}\\
\nonumber
&a^{-1}uc_{m-3}\xrightarrow{x}a^{-1}u^xc_{m-4}c_{m-3}\xrightarrow{y}abu^{xy}\xrightarrow{z}au^xc_{m-4}c_{m-3}\\
\nonumber
&au^xc_{m-4}c_{m-3}\xrightarrow{x}auc_{m-3}\xrightarrow{y}au^yc_{m-6}c_{m-5}c_{m-3}\xrightarrow{z}abuc_{m-3}\\
\nonumber
&abuc_{m-3}\xrightarrow{x}a^2bu^xc_{m-4}c_{m-3}\xrightarrow{y}a^2bu^{xy}c_{m-4}c_{m-3}\xrightarrow{z}a^2u^xc_{m-4}c_{m-3}\\
\nonumber
&a^2u^xc_{m-4}c_{m-3}\xrightarrow{x}uc_{m-3}\xrightarrow{y}u^yc_{m-6}c_{m-5}c_{m-4}\xrightarrow{z}uc_{m-3},
\end{align}
and
\begin{align}\label{eq4}
&a^2u\xrightarrow{x}a^2u^x\xrightarrow{y}bu^{xy}\xrightarrow{z}bu^x
\xrightarrow{x}abu\xrightarrow{y}a^{-1}u^yc_{m-4}c_{m-3}\xrightarrow{z}a^{-1}u\\\nonumber
&a^{-1}u\xrightarrow{x}au^x\xrightarrow{y}au^{xy}c_{m-4}c_{m-3}\xrightarrow{z}abu^x
\xrightarrow{x}bu\xrightarrow{y}a^2u^y\xrightarrow{z}a^2bu\\\nonumber
&a^2bu\xrightarrow{x}a^{-1}bu^x\xrightarrow{y}a^{-1}bu^{xy}c_{m-4}c_{m-3}\xrightarrow{z}a^{-1}bu^x
\xrightarrow{x}a^2bu\xrightarrow{y}a^2bu^y\xrightarrow{z}a^2u.
\end{align}

\begin{lemma}\label{lem6}
Suppose $m$ is odd. Then the permutation $(xyz)^8$ of $H$ has cycle decomposition
\[
(xyz)^8=\left(\prod_{u\in U}(a^2u,a^{-1}u,a^2bu)\right)\left(\prod_{u\in U}(bu,abu,a^{-1}bu)\right).
\]
\end{lemma}

\begin{proof}
Denote $H_{i,j,k}=\{a^ib^juc_{m-3}^k\mid u\in U\}$, where $i\in\{-1,0,1,2\}$ and $j,k\in\{0,1\}$. Then $\{H_{i,j,k}\mid-1\leqslant i\leqslant2,\ 0\leqslant j\leqslant1,\ 0\leqslant k\leqslant1\}$ forms a partition of $H$.

By~\eqref{eq1}~and~\eqref{eq2}, $u^{(xyz)^2}=u$ and $(au)^{(xyz)^2}=au$ for all $u\in U$. Hence $(xyz)^8$ fixes $H_{0,0,0}\cup H_{1,0,0}$ pointwise. From~\eqref{eq3} one sees that $(xyz)^8$ fixes $uc_{m-3}$, $bu^xc_{m-4}c_{m-3}$, $a^{-1}buc_{m-3}$, $a^2bu^xc_{m-4}c_{m-3}$, $a^{-1}uc_{m-3}$, $au^xc_{m-4}c_{m-3}$, $abuc_{m-3}$ and $a^2u^xc_{m-4}c_{m-3}$ for all $u\in U$. Noting $\{u^xc_{m-4}c_{m-3}\mid u\in U\}=Uc_{m-3}$, we conclude that $(xyz)^8$ fixes
\[
H_{0,0,1}\cup H_{0,1,1}\cup H_{-1,1,1}\cup H_{2,1,1}\cup H_{-1,0,1}\cup H_{1,0,1}\cup H_{1,1,1}\cup H_{2,0,1}
\]
pointwise. This together with the conclusion that $(xyz)^8$ fixes $H_{0,0,0}\cup H_{1,0,0}$ pointwise shows it suffices to prove that $(xyz)^8$ induces the permutation
\[
\left(\prod_{u\in U}(a^2u,a^{-1}u,a^2bu)\right)\left(\prod_{u\in U}(bu,abu,a^{-1}bu)\right)
\]
on $H_{2,0,0}\cup H_{-1,0,0}\cup H_{2,1,0}\cup H_{0,1,0}\cup H_{1,1,0}\cup H_{-1,1,0}$. Indeed, \eqref{eq4} implies
\[
a^2u\xrightarrow{(xyz)^8}a^{-1}u\xrightarrow{(xyz)^8}a^2bu\xrightarrow{(xyz)^8}a^2u
\]
and
\[
bu^x\xrightarrow{(xyz)^8}abu^x\xrightarrow{(xyz)^8}a^{-1}bu^x\xrightarrow{(xyz)^8}bu^x
\]
for all $u\in U$. This yields the desired conclusion since $x|_U\in\Aut(U)$.
\end{proof}

Before proving the next lemma, we note that for each $u\in U$,

\begin{equation}\label{eq5}
a^{-1}bu\xrightarrow{xyzx}a^2u\xrightarrow{xyzx}abu\xrightarrow{yz}a^{-1}u
\xrightarrow{xyzx}bu\xrightarrow{yz}a^2bu\xrightarrow{zyxzyxyz}au
\end{equation}
and
\begin{align}\label{eq6}
&a^{-1}uc_{m-3}\xrightarrow{yz}auc_{m-3}\xrightarrow{xyzx}a^2buc_{m-3}\xrightarrow{yz}a^2uc_{m-3}\xrightarrow{yz}buc_{m-3}\\\nonumber
&auc_{m-3}\xrightarrow{yz}abuc_{m-3}\xrightarrow{xyzx}uc_{m-3}\xrightarrow{xyzx}a^{-1}buc_{m-3}.
\end{align}

\begin{lemma}\label{lem4}
Suppose $m$ is odd. Then for each $g\in H\setminus U$, there exists $\zeta\in\langle x,y,z\rangle$ such that $g^\zeta=c_{m-3}$.
\end{lemma}

\begin{proof}
Let $v\in\langle c_1,c_2,\dots,c_{m-3}\rangle$ such that $g\in\langle a,b\rangle v$, and define $\chi$, $\psi$ and $\omega$ as in Lemma~\ref{lem3} with $\ell=m-3$ and $e_i=c_i$. Then $\chi$, $\psi$ and $\psi\omega$ are all involutions in $\Sym(\langle c_1,c_2,\dots,c_{m-3}\rangle)$, and by Lemma~\ref{lem3}, there exist $\eta_1,\eta_2,\dots,\eta_t\in\{\chi,\psi,\psi\omega\}$ such that $v^{\eta_1\eta_2\cdots\eta_t}=c_{m-3}$. Let $\eta_0=1\in\Sym(\langle c_1,c_2,\dots,c_{m-3}\rangle)$ and $\zeta_0=1\in\langle x,y,z\rangle$. Obviously, $g^{\zeta_0}\in\langle a,b\rangle v^{\eta_0}\setminus U$. We shall prove by induction that there exist $\zeta_0,\zeta_1,\dots,\zeta_t\in\langle x,y,z\rangle$ with
\[
g^{\zeta_0\zeta_1\dots\zeta_t}\in\langle a,b\rangle v^{\eta_0\eta_1\dots\eta_t}\setminus U.
\]
By~\eqref{eq5}, for each $u\in U$ and $\alpha,\beta\in\langle a,b\rangle^*$, there exists $\varepsilon\in\langle x,yz\rangle$ such that $(\alpha u)^\varepsilon=\beta u$. By~\eqref{eq6}, for each $u\in U$ and $\alpha,\beta\in\langle a,b\rangle$, there exists $\varepsilon\in\langle x,yz\rangle$ such that $(\alpha uc_{m-3})^\varepsilon=\beta uc_{m-3}$.

Suppose that there exist $\zeta_0,\zeta_1,\dots,\zeta_{i-1}\in\langle x,y,z\rangle$ with $1\leqslant i\leqslant t$ and
\[
g^{\zeta_0\zeta_1\dots\zeta_{i-1}}\in\langle a,b\rangle v^{\eta_0\eta_1\dots\eta_{i-1}}\setminus U.
\]
If $\eta_i=\chi$, then let $\zeta_i=x$. If $\eta_i=\psi$, then there exists $\varepsilon_i\in\langle x,yz\rangle$ such that $(g^{\zeta_0\zeta_1\dots\zeta_{i-1}})^{\varepsilon_i}=a^2v^{\eta_0\eta_1\dots\eta_{i-1}}$ and we let $\zeta_i=\varepsilon_i y$.
If $\eta_i=\psi\omega$, then there exists $\varepsilon_i\in\langle x,yz\rangle$ such that $(g^{\zeta_0\zeta_1\dots\zeta_{i-1}})^{\varepsilon_i}=av^{\eta_0\eta_1\dots\eta_{i-1}}$ and we let $\zeta_i=\varepsilon_i y$. It follows that
\[
g^{\zeta_0\zeta_1\dots\zeta_{i-1}\zeta_i}=(g^{\zeta_0\zeta_1\dots\zeta_{i-1}})^{\zeta_i}
\in\left(\langle a,b\rangle v^{\eta_0\eta_1\dots\eta_{i-1}}\setminus U\right)^{\zeta_i}
\subseteq\langle a,b\rangle v^{\eta_0\eta_1\dots\eta_i}\setminus U.
\]

By induction we now have $\zeta_0,\zeta_1,\dots,\zeta_t\in\langle x,y,z\rangle$ such that
\[
g^{\zeta_0\zeta_1\dots\zeta_t}\in\langle a,b\rangle v^{\eta_0\eta_1\dots\eta_t}\setminus U=\langle a,b\rangle c_{m-3}.
\]
Then as there exists $\varepsilon\in\langle x,yz\rangle$ such that $(g^{\zeta_0\zeta_1\dots\zeta_t})^\varepsilon=c_{m-3}$, we have $g^\zeta=c_{m-3}$ with $\zeta:=\zeta_0\zeta_1\dots\zeta_t\varepsilon\in\langle x,y,z\rangle$.
\end{proof}

\begin{lemma}\label{lem5}
Suppose that $m$ is odd. Then $\langle x,y,z\rangle$ is transitive on $H^*$.
\end{lemma}

\begin{proof}
In view of Lemma~\ref{lem4}, we only need to prove that for each $u\in U^*$, there exists $\varepsilon\in\langle x,y\rangle$ such that $u^\varepsilon\in Uc_{m-3}$. Write $u=c_1^{k_1}c_2^{k_2}\cdots c_{m-3}^{k_{m-3}}$ with $k_1,k_2,\dots,k_{m-3}\in\{0,1\}$. Denote $V_i=\langle c_i,c_{i+1},\dots,c_{m-3}\rangle$ for $1\leqslant i\leqslant m-3$, and set $V_{m-2}=1$. Let $s$ be the smallest integer in $\{0,1,\dots,(m-5)/2\}$ such that $k_{2s+1}+k_{2s+2}>0$. Taking
\[
\varepsilon=x^{k_{2s+1}+k_{2s+2}-1}y(xy)^s,
\]
we prove below that $u^\varepsilon\in Uc_{m-3}$ by induction on $s$.

First suppose $s=0$. If $k_1=0$ and $k_2=1$, then $u=c_2u_1$ with $u_1\in V_3$ and $u_1\in Uc_{m-3}$ since $u\in U$. In this case, $u_1^y\in Uc_{m-3}$, and it follows that $u^\varepsilon=(c_2u_1)^y=c_1u_1^y\in Uc_{m-3}$. If $k_1=1$ and $k_2=0$, then $u=c_1u_1$ with $u_1\in V_3$ and $u_1\in U$ as $u\in U$. In this case, $u_1^y\in U$ and so $u^\varepsilon=(c_1u_1)^y=c_2u_1^y\in Uc_{m-3}$. If $k_1=k_2=1$, then $u=c_1c_2u_1$ with $u_1\in V_3$ and $u_1\in Uc_{m-3}$ since $u\in U$. In this case, $u_1^x=a^2u_2$ for some $u_2\in V_3\cap Uc_{m-3}$, whence $u^\varepsilon=(c_1c_2u_1)^{xy}=(a^2c_2u_1^x)^y=(c_2u_2)^y=c_1u_2^y\in Uc_{m-3}$ as $u_2^y\in Uc_{m-3}$.

Next suppose $s>0$. If $k_{2s+1}=0$ and $k_{2s+2}=1$, then $u=c_{2s+2}u_1$ with $u_1\in V_{2s+3}$, which implies $u^y=(c_{2s+2}u_1)^y=c_{2s-1}c_{2s}c_{2s+1}u_1^y\in c_{2s-1}c_{2s} V_{2s+1}$. If $k_{2s+1}=1$ and $k_{2s+2}=0$, then $u=c_{2s+1}u_1$ with $u_1\in V_{2s+3}$ and therefore $u^y=(c_{2s+1}u_1)^y=c_{2s-1}c_{2s}c_{2s+2}u_1^y\in c_{2s-1}c_{2s}V_{2s+1}$. If $k_{2s+1}=k_{2s+2}=1$, then $u=c_{2s+1}c_{2s+2}u_1$ with $u_1\in V_{2s+3}$ and $u_1\in Uc_{m-3}$ since $u\in U$. In this case, $u_1^x=a^2u_2$ for some $u_2\in V_{2s+3}$, and so $u^{xy}=(c_{2s+1}c_{2s+2}u_1)^{xy}=(a^2c_{2s+2}u_1^x)^y=(c_{2s+2}u_2)^y=c_{2s-1}c_{2s}c_{2s+1}u_2^y\in c_{2s-1}c_{2s}V_{2s+1}$. To sum up, we always have $u^{\varepsilon_0}\in c_{2s-1}c_{2s}V_{2s+1}$, where
\[
\varepsilon_0=x^{k_{2s+1}+k_{2s+2}-1}y.
\]
By the inductive hypothesis,
\[
(u^{\varepsilon_0})^{(xy)^s}=(u^{\varepsilon_0})^{x^{1+1-1}y(xy)^{s-1}}\in Uc_{m-3}.
\]
Consequently,
\[
u^{x^{k_{2s+1}+k_{2s+2}-1}y(xy)^s}=(u^{\varepsilon_0})^{(xy)^s}\in Uc_{m-3},
\]
completing the proof.
\end{proof}

\begin{lemma}\label{lem1}
Suppose that $m$ is odd. Then $\langle x,y,R(H)\rangle=\Alt(H)$.
\end{lemma}

\begin{proof}
Let $G=\langle x,y,R(H)\rangle$. Notice that $x$, $y$ and $z$ are all involutions of $G$ fixing $1$. By Lemma~\ref{lem5}, $\langle x,y,z\rangle$ is transitive on $H^*$, and so is $G_1$, the stabilizer of $1$ in $G$. This together with the transitivity of $R(H)$ on $H$ implies that $G$ is doubly transitive on $H$. Therefore, one of cases~(i)--(iii) in Lemma~\ref{lem2} holds.

Assume that $G\leqslant\AGL_m(2)$ as in case~(i) of Lemma~\ref{lem2}. Then $(xyz)^8\in G_1\leqslant\GL_m(2)$ and hence the set of fixed points of $(xyz)^8$ is a vector space over $\mathbb{F}_2$. However, Lemma~\ref{lem6} shows that the number of fixed points of $(xyz)^8$ is $|H|-3|U|-3|U|=5\cdot2^{m-3}$, a contradiction.

Assume that $\PSL_2(q)\leqslant G\leqslant\PGaL_2(q)$ as in case~(ii) of Lemma~\ref{lem2}, where $q=2^m-1$ is a prime. Then $R(H)\leqslant G\leqslant\PGaL_2(q)=\PGL_2(q)$. It follows that $R(H)$ is contained in $\D_{2(q+1)}$, the Sylow $2$-subgroup of $\PGL_2(q)$. This is impossible since $R(H)\cong\D_8\times\ZZ_2^{m-3}$.

Now $\Alt(H)\leqslant G\leqslant\Sym(H)$. This in conjunction with Lemma~\ref{lem7} forces $G=\Alt(H)$, which completes the proof.
\end{proof}

\subsection{Even $m$}

Throughout this subsection, let $m$ be even, $h_1=hc_{m-3}$, $H_1=\langle a,b,c_1,c_2,\dots,c_{m-4}\rangle$, $K_1=\langle a^2,b,c_1,c_2,\dots,c_{m-4}\rangle$ and
\[
U=\left\{\prod_{i=1}^{m-3}c_i^{k_i}\mid\sum_{j=1}^{(m-4)/2}k_{2j}\equiv k_{m-3}\pmod{2}\right\}.
\]
Define permutations $x_1$, $y_1$ and $z_1$ on $H_1$ such that $x_1=x|_{H_1}$, $y_1|_{K_1}=y|_{K_1}$, $y_1|_{h_1K_1}=(yR(c_{m-3}))|_{h_1K_1}$, $z_1|_{K_1}=z|_{K_1}$ and $z_1|_{h_1K_1}=(zR(c_{m-3}))|_{h_1K_1}$. One can verify readily that $(h_1k_1)^{y_1}=h_1k_1^{y_1}$ for all $k_1\in K_1$, $z_1=(R(h_1)y_1R(h_1^{-1}))|_{H_1}$ and $y_1z_1=(yz)|_{H_1}$. For each $u\in U\cap H_1$ we have
\begin{align}\label{eq7}
&uc_{m-3}\xrightarrow{xyz}bu^xc_{m-3}\xrightarrow{xyz}a^{-1}buc_{m-3}\xrightarrow{xyz}a^2bu^xc_{m-3}\xrightarrow{xyz}a^{-1}uc_{m-3}\\\nonumber
&a^{-1}uc_{m-3}\xrightarrow{xyz}au^xc_{m-3}\xrightarrow{xyz}abuc_{m-3}\xrightarrow{xyz}a^2u^xc_{m-3}\xrightarrow{xyz}uc_{m-3},
\end{align}
\begin{equation}\label{eq8}
uc_{m-4}c_{m-3}\xrightarrow{xyz}u^xc_{m-5}c_{m-4}c_{m-3}\xrightarrow{xyz}uc_{m-4}c_{m-3},
\end{equation}
\begin{equation}\label{eq9}
auc_{m-4}c_{m-3}\xrightarrow{xyz}au^xc_{m-5}c_{m-4}c_{m-3}\xrightarrow{xyz}auc_{m-4}c_{m-3}
\end{equation}
and
\begin{align}\label{eq10}
&a^2uc_{m-4}c_{m-3}\xrightarrow{xyz}bu^xc_{m-5}c_{m-4}c_{m-3}\xrightarrow{xyz}a^{-1}uc_{m-4}c_{m-3}\\\nonumber
&a^{-1}uc_{m-4}c_{m-3}\xrightarrow{xyz}abu^xc_{m-5}c_{m-4}c_{m-3}\xrightarrow{xyz}a^2buc_{m-4}c_{m-3}\\\nonumber
&a^2buc_{m-4}c_{m-3}\xrightarrow{xyz}a^{-1}bu^xc_{m-5}c_{m-4}c_{m-3}\xrightarrow{xyz}a^2uc_{m-4}c_{m-3}.
\end{align}

\begin{lemma}\label{lem11}
Suppose that $m$ is even. Then the permutation $(xyz)^8|_{H_1c_{m-3}}$ has cycle decomposition
\begin{align*}
(xyz)^8|_{H_1c_{m-3}}&=\left(\prod_{u\in U\cap H_1}(a^2uc_{m-4}c_{m-3},a^{-1}uc_{m-4}c_{m-3},a^2buc_{m-4}c_{m-3})\right)\\
&\times\left(\prod_{u\in U\cap H_1}(buc_{m-4}c_{m-3},abuc_{m-4}c_{m-3},a^{-1}buc_{m-4}c_{m-3})\right).
\end{align*}
\end{lemma}

\begin{proof}
Denote $H_{i,j,k}=\{a^ib^juc_{m-4}^kc_{m-3}\mid u\in U\cap H_1\}$, where $i\in\{-1,0,1,2\}$ and $j,k\in\{0,1\}$. Then $\{H_{i,j,k}\mid-1\leqslant i\leqslant2,\ 0\leqslant j\leqslant1,\ 0\leqslant k\leqslant1\}$ forms a partition of $H_1c_{m-3}$.

By~\eqref{eq8}~and~\eqref{eq9}, $(xyz)^2$ fixes $H_{0,0,1}\cup H_{1,0,1}$ pointwise and so does $(xyz)^8$. From~\eqref{eq7} one sees that $(xyz)^8$ fixes $uc_{m-3}$, $bu^xc_{m-3}$, $a^{-1}buc_{m-3}$, $a^2bu^xc_{m-3}$, $a^{-1}uc_{m-3}$, $au^xc_{m-3}$, $abuc_{m-3}$ and $a^2u^xc_{m-3}$ for all $u\in U\cap H_1$. As $\{u^x\mid u\in U\cap H_1\}=U\cap H_1$, we then conclude that $(xyz)^8$ fixes
\[
H_{0,0,0}\cup H_{0,1,0}\cup H_{-1,1,0}\cup H_{2,1,0}\cup H_{-1,0,0}\cup H_{1,0,0}\cup H_{1,1,0}\cup H_{2,0,0}
\]
pointwise. This together with the conclusion that $(xyz)^8$ fixes $H_{0,0,1}\cup H_{1,0,1}$ pointwise shows it suffices to prove that $(xyz)^8$ induces the permutation
\[
\prod_{u\in U\cap H_1}(a^2uc_{m-4}c_{m-3},a^{-1}uc_{m-4}c_{m-3},a^2buc_{m-4}c_{m-3})
\]
on $H_{2,0,1}\cup H_{-1,0,1}\cup H_{2,1,1}$ and the permutation
\[
\prod_{u\in U\cap H_1}(buc_{m-4}c_{m-3},abuc_{m-4}c_{m-3},a^{-1}buc_{m-4}c_{m-3})
\]
on $H_{0,1,1}\cup H_{1,1,1}\cup H_{-1,1,1}$. Indeed, \eqref{eq10} implies
\[
a^2uc_{m-4}c_{m-3}\xrightarrow{(xyz)^8}a^{-1}uc_{m-4}c_{m-3}\xrightarrow{(xyz)^8}
a^2buc_{m-4}c_{m-3}\xrightarrow{(xyz)^8}a^2uc_{m-4}c_{m-3}
\]
and
\begin{align*}
&bu^xc_{m-5}c_{m-4}c_{m-3}\xrightarrow{(xyz)^8}abu^xc_{m-5}c_{m-4}c_{m-3}\xrightarrow{(xyz)^8}a^{-1}bu^xc_{m-5}c_{m-4}c_{m-3}\\
&a^{-1}bu^xc_{m-5}c_{m-4}c_{m-3}\xrightarrow{(xyz)^8}bu^xc_{m-5}c_{m-4}c_{m-3}
\end{align*}
for all $u\in U\cap H_1$. This yields the desired conclusion since the maps $u\mapsto u^x$ and $u\mapsto u^xc_{m-5}$ are both bijections from $U\cap H_1$ onto itself.
\end{proof}

\begin{lemma}\label{lem9}
Suppose that $m$ is even. For each $v\in U$ and $\alpha,\beta\in\langle a,b\rangle^*$, there exists $\varepsilon\in\langle x,yz\rangle$ such that $(\alpha v)^\varepsilon=\beta v$. For each $v\in\langle c_1,c_2,\dots,c_{m-3}\rangle\setminus U$ and $\alpha,\beta\in\langle a,b\rangle$, there exists $\varepsilon\in\langle x,yz\rangle$ such that $(\alpha v)^\varepsilon=\beta v$.
\end{lemma}

\begin{proof}
First recall from~\eqref{eq5} that for each $u\in U\cap H_1$ and $\alpha,\beta\in\langle a,b\rangle^*$, there exist integers $k_1,k_2,\dots,k_{2r-1},k_{2r}$ with
\[
(\alpha u)^{\prod_{i=1}^rx_1^{k_{2i-1}}(y_1z_1)^{k_{2i}}}=\beta u.
\]
Then taking $\varepsilon=\prod_{i=1}^rx^{k_{2i-1}}(yz)^{k_{2i}}$, we have $(\alpha u)^\varepsilon=\beta u$ since $x|_{H_1}=x_1$ and $(yz)|_{H_1}=y_1z_1$. Next note that for $u\in U\cap H_1$,
\begin{align*}
&buc_{m-4}c_{m-3}\xrightarrow{yz}a^2buc_{m-4}c_{m-3}\xrightarrow{yz}a^2uc_{m-4}c_{m-3}\xrightarrow{(xyz)^2}a^{-1}uc_{m-4}c_{m-3}\\
\nonumber
&auc_{m-4}c_{m-3}\xrightarrow{zy}a^{-1}uc_{m-4}c_{m-3}\xrightarrow{zy}abuc_{m-4}c_{m-3}\xrightarrow{(xyz)^2}a^{-1}buc_{m-4}c_{m-3}.
\end{align*}
We then conclude that for each $v\in U$ and $\alpha,\beta\in\langle a,b\rangle^*$, there exists $\varepsilon\in\langle x,yz\rangle$ such that $(\alpha v)^\varepsilon=\beta v$.

Similarly, one derives from~\eqref{eq6} that for each $u\in U\cap H_1$ and $\alpha,\beta\in\langle a,b\rangle$, there exists $\varepsilon\in\langle x,yz\rangle$ with $(\alpha uc_{m-4})^\varepsilon=\beta uc_{m-4}$. Moreover,
\begin{align*}
&a^{-1}uc_{m-3}\xrightarrow{yz}auc_{m-3}\xrightarrow{xyzx}a^2buc_{m-3}\xrightarrow{yz}a^2uc_{m-3}\xrightarrow{yz}buc_{m-3}\\\nonumber
&auc_{m-3}\xrightarrow{yz}abuc_{m-3}\xrightarrow{xyzx}uc_{m-3}\xrightarrow{xyzx}a^{-1}buc_{m-3}
\end{align*}
for all $u\in U\cap H_1$. Hence for each $v\in\langle c_1,c_2,\dots,c_{m-3}\rangle\setminus U$ and $\alpha,\beta\in\langle a,b\rangle$, there exists $\varepsilon\in\langle x,yz\rangle$ such that $(\alpha v)^\varepsilon=\beta v$.
\end{proof}

\begin{lemma}\label{lem8}
Suppose that $m$ is even. Then for each $g\in H\setminus U$, there exists $\zeta\in\langle x,y,z\rangle$ such that $g^\zeta=h$.
\end{lemma}

\begin{proof}
Let $v\in\langle c_1,c_2,\dots,c_{m-4}\rangle$ such that $g\in\langle a,b,c_{m-3}\rangle v$, and define $\chi$, $\psi$ and $\omega$ as in Lemma~\ref{lem3} with $\ell=m-4$ and $e_i=c_i$. Then $\chi$, $\psi$ and $\psi\omega$ are all involutions in $\Sym(\langle c_1,c_2,\dots,c_{m-4}\rangle)$, and by Lemma~\ref{lem3}, there exist $\eta_1,\eta_2,\dots,\eta_t\in\{\chi,\psi,\psi\omega\}$ such that $v^{\eta_1\eta_2\cdots\eta_t}=a^{-1}h_1$. Let $\eta_0=1\in\Sym(\langle c_1,c_2,\dots,c_{m-4}\rangle)$ and $\zeta_0=1\in\langle x,y,z\rangle$. Obviously, $g^{\zeta_0}\in\langle a,b,c_{m-3}\rangle v^{\eta_0}\setminus U$. We shall prove by induction that there exist $\zeta_0,\zeta_1,\dots,\zeta_t\in\langle x,y,z\rangle$ with
\[
g^{\zeta_0\zeta_1\dots\zeta_t}\in\langle a,b\rangle v^{\eta_0\eta_1\dots\eta_t}\setminus U.
\]

Suppose there exist $\zeta_0,\zeta_1,\dots,\zeta_{i-1}\in\langle x,y,z\rangle$ for some $i\in\{1,\dots,t\}$ such that
\[
g^{\zeta_0\zeta_1\dots\zeta_{i-1}}\in\langle a,b,c_{m-3}\rangle v^{\eta_0\eta_1\dots\eta_{i-1}}\setminus U.
\]
If $\eta_i=\chi$, then let $\zeta_i=x$. If $\eta_i=\psi$, then by Lemma \ref{lem9} there exists $\varepsilon_i\in\langle x,yz\rangle$ such that $(g^{\zeta_0\zeta_1\dots\zeta_{i-1}})^{\varepsilon_i}\in a^2\langle c_{m-3}\rangle v^{\eta_0\eta_1\dots\eta_{i-1}}$ and we let $\zeta_i=\varepsilon_i y$.
If $\eta_i=\psi\omega$, then by Lemma \ref{lem9} there exists $\varepsilon_i\in\langle x,yz\rangle$ such that $(g^{\zeta_0\zeta_1\dots\zeta_{i-1}})^{\varepsilon_i}\in a\langle c_{m-3}\rangle v^{\eta_0\eta_1\dots\eta_{i-1}}$ and we let $\zeta_i=\varepsilon_i y$. It follows that
\[
g^{\zeta_0\zeta_1\dots\zeta_{i-1}\zeta_i}=(g^{\zeta_0\zeta_1\dots\zeta_{i-1}})^{\zeta_i}
\in\left(\langle a,b,c_{m-3}\rangle v^{\eta_0\eta_1\dots\eta_{i-1}}\setminus U\right)^{\zeta_i}
\subseteq\langle a,b,c_{m-3}\rangle v^{\eta_0\eta_1\dots\eta_i}\setminus U.
\]

By induction we now have $\zeta_0,\zeta_1,\dots,\zeta_t\in\langle x,y,z\rangle$ such that
\[
g^{\zeta_0\zeta_1\dots\zeta_t}\in\langle a,b,c_{m-3}\rangle v^{\eta_0\eta_1\dots\eta_t}\setminus U=\langle a,b,c_{m-3}\rangle a^{-1}h_1\setminus U.
\]
Then by Lemma \ref{lem9} there exists $\varepsilon\in\langle x,yz\rangle$ such that $(g^{\zeta_0\zeta_1\dots\zeta_t})^\varepsilon\in h_1\langle c_{m-3}\rangle$. Let $\zeta =\zeta_0\zeta_1\dots\zeta_{r}\varepsilon y$ if $(g^{\zeta_0\zeta_1\dots\zeta_{r}})^{\varepsilon}=h_1$ and let $\zeta =\zeta_0\zeta_1\dots\zeta_{r}\varepsilon $ if $(g^{\zeta_0\zeta_1\dots\zeta_{r}})^{\varepsilon}=h_1c_{m-3}$. Then $\zeta\in\langle x,y,z\rangle$, and in view of $h=h_1c_{m-3}=h_1^y$ we have $g^\zeta=h$.
\end{proof}

\begin{lemma}\label{lem10}
Suppose that $m$ is even. Then $\langle x,y,z\rangle$ is transitive on $H^*$.
\end{lemma}

\begin{proof}
Due to Lemma~\ref{lem9}, we only need to prove that for each $u\in U^*$, there exists $\varepsilon\in\langle x,y\rangle$ such that $u^\varepsilon\in Uc_{m-3}$. Write $u=c_1^{k_1}c_2^{k_2}\cdots c_{m-3}^{k_{m-3}}$ with $k_1,k_2,\dots,k_{m-3}\in\{0,1\}$. Denote $V_i=\langle c_i,c_{i+1},\dots,c_{m-3}\rangle$ for $1\leqslant i\leqslant m-3$. Since $u\neq c_{m-3}$, there exists $0\leqslant j\leqslant(m-6)/2$ such that $k_{2j+1}+k_{2j+2}>0$. Let $s$ be the smallest integer in $\{0,1,\dots,(m-6)/2\}$ such that $k_{2s+1}+k_{2s+2}>0$. Taking
\[
\varepsilon=x^{k_{2s+1}+k_{2s+2}-1}y(xy)^s,
\]
we prove below that $u^\varepsilon\in Uc_{m-3}$ by induction on $s$.

First suppose $s=0$. If $k_1=0$ and $k_2=1$, then $u=c_2u_1$ with $u_1\in V_3$ and $u_1\in Uc_{m-3}$ since $u\in U$. In this case, $u_1^y\in Uc_{m-3}$, and it follows that $u^\varepsilon=(c_2u_1)^y=c_1u_1^y\in Uc_{m-3}$. If $k_1=1$ and $k_2=0$, then $u=c_1u_1$ with $u_1\in V_3$ and $u_1\in U$ as $u\in U$. In this case, $u_1^y\in U$ and so $u^\varepsilon=(c_1u_1)^y=c_2u_1^y\in Uc_{m-3}$. If $k_1=k_2=1$, then $u=c_1c_2u_1$ with $u_1\in V_3$ and $u_1\in Uc_{m-3}$ since $u\in U$. In this case, $u_1^x=a^2u_2$ for some $u_2\in V_3\cap Uc_{m-3}$, whence $u^\varepsilon=(c_1c_2u_1)^{xy}=(a^2c_2u_1^x)^y=(c_2u_2)^y=c_1u_2^y\in Uc_{m-3}$ as $u_2^y\in Uc_{m-3}$.

Next suppose $s>0$. If $k_{2s+1}=0$ and $k_{2s+2}=1$, then $u=c_{2s+2}u_1$ with $u_1\in V_{2s+3}$, which implies $u^y=(c_{2s+2}u_1)^y=c_{2s-1}c_{2s}c_{2s+1}u_1^y\in c_{2s-1}c_{2s} V_{2s+1}$. If $k_{2s+1}=1$ and $k_{2s+2}=0$, then $u=c_{2s+1}u_1$ with $u_1\in V_{2s+3}$ and therefore $u^y=(c_{2s+1}u_1)^y=c_{2s-1}c_{2s}c_{2s+2}u_1^y\in c_{2s-1}c_{2s}V_{2s+1}$. If $k_{2s+1}=k_{2s+2}=1$, then $u=c_{2s+1}c_{2s+2}u_1$ with $u_1\in V_{2s+3}$ and $u_1\in Uc_{m-3}$ since $u\in U$. In this case, $u_1^x=a^2u_2$ for some $u_2\in V_{2s+3}$, and so $u^{xy}=(c_{2s+1}c_{2s+2}u_1)^{xy}=(a^2c_{2s+2}u_1^x)^y=(c_{2s+2}u_2)^y=c_{2s-1}c_{2s}c_{2s+1}u_2^y\in c_{2s-1}c_{2s}V_{2s+1}$. To sum up, we always have $u^{\varepsilon_0}\in c_{2s-1}c_{2s}V_{2s+1}$, where
\[
\varepsilon_0=x^{k_{2s+1}+k_{2s+2}-1}y.
\]
By the inductive hypothesis,
\[
(u^{\varepsilon_0})^{(xy)^s}=(u^{\varepsilon_0})^{x^{1+1-1}y(xy)^{s-1}}\in Uc_{m-3}.
\]
Consequently,
\[
u^{x^{k_{2s+1}+k_{2s+2}-1}y(xy)^s}=(u^{\varepsilon_0})^{(xy)^s}\in Uc_{m-3},
\]
completing the proof.
\end{proof}

\begin{lemma}\label{lem12}
Suppose that $m$ is even. Then $\langle x,y,R(H)\rangle=\Alt(H)$.
\end{lemma}

\begin{proof}
Let $G=\langle x,y,R(H)\rangle$. Notice that $x$, $y$ and $z$ are all involutions of $G$ fixing $1$. By Lemma~\ref{lem10}, $\langle x,y,z\rangle$ is transitive on $H^*$, and so is $G_1$, the stabilizer of $1$ in $G$. This together with the transitivity of $R(H)$ on $H$ implies that $G$ is doubly transitive on $H$. Therefore, either cases~(i) or case~(iii) in Lemma~\ref{lem2} holds since $m$ is even.

Assume that $G\leqslant\AGL_m(2)$ as in case~(i) of Lemma~\ref{lem2}. Then $(xyz)^8\in G_1\leqslant\GL_m(2)$ and hence the set of fixed points of $(xyz)^8$ is a vector space over $\mathbb{F}_2$. Since $x|_{H_1}=x_1$ and $(yz)|_{H_1}=y_1z_1$, we have $(xyz)^8|_{H_1}=(x_1y_1z_1)^8$, and thus Lemma~\ref{lem6} shows that the number of fixed points of $(xyz)^8|_{H_1}$ is $|H_1|-3|U\cap H_1|-3|U\cap H_1|=5\cdot2^{m-4}$. Furthermore, the number of fixed points of $(xyz)^8|_{H_1c_{m-3}}$ is $|H_1c_{m-3}|-3|U\cap H_1|-3|U\cap H_1|=5\cdot2^{m-4}$. Hence the number of fixed points of $(xyz)^8$ is $5\cdot2^{m-4}+5\cdot2^{m-4}=5\cdot2^{m-3}$, a contradiction.

Now $\Alt(H)\leqslant G\leqslant\Sym(H)$. This in conjunction with Lemma~\ref{lem7} forces $G=\Alt(H)$, which completes the proof.
\end{proof}

\section{Proof of Theorem~\ref{thm1}}

Theorem~\ref{thm1} will follow directly from Lemmas~\ref{lem14}--\ref{lem15}.

\begin{lemma}\label{lem14}
$\Ga_m$ is a connected cubic graph.
\end{lemma}

\begin{proof}
By Lemmas~\ref{lem1} and~\ref{lem12} we have for each $m\geqslant4$ that
\begin{equation}\label{eq11}
\langle x,y,R(H)\rangle=\Alt(H),
\end{equation}
which already implies the connectivity of $\Ga_m$. To prove that $\Ga_m$ is cubic, we show $|R(H)\{x,y\}R(H)|=3|R(H)|$ in the following.

It is straightforward to verify that $x$ and $y$ are involutions normalizing $R(H)$ and $R(K)$, respectively. As a consequence, $R(H)xR(H)=R(H)x$ and $yR(H)y\cap R(H)\geqslant yR(K)y\cap R(K)=R(K)$. Notice that $yR(H)y\cap R(H)\neq R(H)$ for otherwise
\[
\langle x,y,R(H)\rangle\leqslant\Nor_{\Alt(H)}(R(H))<\Alt(H),
\]
contrary to~\eqref{eq11}. We derive that $yR(H)y\cap R(H)=R(K)$ as $R(K)$ has index $2$ in $R(H)$. Accordingly, $|R(H)yR(H)|/|R(H)|=|R(H)|/|yR(H)y\cap R(H)|=2$. Moreover, $R(H)xR(H)\cap R(H)yR(H)=\emptyset$ for otherwise $y\in\langle x,R(H)\rangle$, which would cause a contradiction $\langle x,y,R(H)\rangle\leqslant\langle x,R(H)\rangle\leqslant\Nor_{\Alt(H)}(R(H))<\Alt(H)$ to~\eqref{eq11}. Hence
\begin{align*}
|R(H)\{x,y\}R(H)|&=|R(H)xR(H)|+|R(H)yR(H)|\\
&=|R(H)x|+2|R(H)|=3|R(H)|,
\end{align*}
as desired.
\end{proof}

\begin{lemma}\label{lem16}
$\Cay(\Alt(H^*),\{x,y,z\})$ is a nonnormal Cayley graph of $\Alt(H^*)$ and is isomorphic to $\Ga_m$ by the map $g\mapsto R(H)g$.
\end{lemma}

\begin{proof}
Let $S=\{x,y,z\}$. Consider the map $\varphi:g\mapsto R(H)g$ from $\Alt(H^*)$ to the vertex set of $\Ga_m$. We see that $\varphi$ is injective as $R(H)\cap\Alt(H^*)=1$, and is therefore bijective as $|\Alt(H^*)|=|\Alt(H)|/|R(H)|$. In particular, $R(H)S$ is a disjoint union of $R(H)x$, $R(H)y$ and $R(H)z$. Then since
\[
R(H)S=R(H)x\cup R(H)y\cup R(H)z\subseteq R(H)\{x,y\}R(H)
\]
and $|R(H)\{x,y\}R(H)|=3|R(H)|$ by Lemma~\ref{lem14}, we conclude that
\[
R(H)S=R(H)\{x,y\}R(H).
\]
For $g_1$ and $g_2$ in $\Alt(H^*)$, $g_1$ is adjacent to $g_2$ in $\Cay(\Alt(H^*),S)$ if and only if $g_2g_1^{-1}\in\{x,y,z\}$, which is equivalent to $R(H)g_2g_1^{-1}\in\{R(H)x,R(H)y,R(H)z\}$. This means that $g_1$ and $g_2$ is adjacent in $\Cay(\Alt(H^*),S)$ if and only if
\[
R(H)g_2g_1^{-1}\subseteq R(H)S=R(H)\{x,y\}R(H),
\]
or equivalently, $R(H)g_1$ is adjacent to $R(H)g_2$ in $\Ga_m$. Therefore, $\varphi$ is a graph isomorphism from $\Cay(\Alt(H^*),S)$ to $\Ga_m$. Moreover, $\Alt(H)$ acts as a group of automorphisms of $\Ga_m$ by right multiplication and $\Alt(H^*)$ is not normal in $\Alt(H)$, whence $\Ga_m$ is a nonnormal Cayley graph of $\Alt(H^*)$.
\end{proof}

\begin{lemma}\label{lem17}
$\Aut(\Alt(H^*),\{x,y,z\})=1$.
\end{lemma}

\begin{proof}
Suppose for a contradiction that there exists $1\neq\sigma\in\Sym(H^*)$ with
\[
\{\sigma^{-1}x\sigma,\sigma^{-1}y\sigma,\sigma^{-1}z\sigma\}=\{x,y,z\}.
\]
Then the conjugation action of $\sigma$ induces a nontrivial permutation of $\{x,y,z\}$ as $\langle x,y,z\rangle=\Alt(H^*)$. By Lemma~\ref{lem16} and~\cite[Theorem~1.1]{XFWX2005} we know that $\Ga_m$ is not arc-transitive. Note that $R(a)$ interchanges the vertices $R(H)y$ and $R(H)z$ of $\Ga_m$. In view of the isomorphism $g\mapsto R(H)g$ in Lemma~\ref{lem16}, we have $\sigma^{-1}x\sigma=x$, $\sigma^{-1}y\sigma=z$ and $\sigma^{-1}z\sigma=y$. For $w\in\Sym(H^*)$, denote by $\Fix(w)$ the set of fixed points of $w$ on $H^*$. It follows that
\begin{equation}\label{eqn-odd}
|\Fix(y)\cap\Fix(xyx)|=|\Fix(y)^\sigma\cap\Fix(xyx)^\sigma|=|\Fix(z)\cap\Fix(xzx)|,
\end{equation}
and
\begin{align}\label{eqn-even}
|\Fix(yxy)\cap\Fix(xyxyx)|&=|\Fix(yxy)^\sigma\cap\Fix(xyxyx)^\sigma|\\\nonumber
&=|\Fix(zxz)\cap\Fix(xzxzx)|.
\end{align}

First assume that $m\equiv1\pmod{4}$. Let
\[
M=\langle a^2b\rangle\times\langle c_1c_2\rangle\times\dots\times\langle c_{2i-1}c_{2i}\rangle\times\dots\times\langle c_{m-4}c_{m-3}\rangle.
\]
It is easy to verify that $\Fix(y)=\{1,h\}M\setminus\{1\}$ and
\[
M^x=\langle a^{-1}b\rangle\times\langle a^2c_2\rangle\times\dots\times\langle a^2c_{2i}\rangle\times\dots\times\langle a^2c_{m-3}\rangle.
\]
This implies that
\begin{align*}
\Fix(y)\cap\Fix(xyx)&=\Fix(y)\cap\Fix(y)^x\\
&=(\{1,h\}M\setminus\{1\})\cap(\{1,h^x\}M^x\setminus\{1\})\\
&=(\{1,h\}M\cap\{1,a^2h\}M^x)\setminus\{1\}\\
&=\emptyset.
\end{align*}
However, since the element $b\prod_{i=1}^{m-3}c_i$ of $H^*$ is fixed by both $z$ and $xzx$, we have $|\Fix(z)\cap\Fix(xzx)|>0$, contrary to~\eqref{eqn-odd}.

Next assume that $m\equiv3\pmod{4}$. Let
\[
M=\langle b\rangle\times\langle c_1c_2\rangle\times\dots\times\langle c_{2i-1}c_{2i}\rangle\times\dots\times\langle c_{m-4}c_{m-3}\rangle.
\]
It is easy to verify that $\Fix(z)=\{1,a^2h\}M\setminus\{1\}$ and
\[
M^x=\langle ab\rangle\times\langle a^2c_2\rangle\times\dots\times\langle a^2c_{2i}\rangle\times\dots\times\langle a^2c_{m-3}\rangle.
\]
Hence $\Fix(z)\cap\Fix(xzx)=\Fix(z)\cap\Fix(z)^x=\emptyset$. However, $a^2b\prod_{i=1}^{m-3}c_i$ is fixed by both $y$ and $xyx$, so $|\Fix(y)\cap\Fix(xyx)|>0$. This again contradicts~\eqref{eqn-odd}.

Now assume that $m\equiv2\pmod{4}$. Let
\[
M=\langle c_1\rangle\times\dots\times\langle c_{2i-1}\rangle\times\dots\times\langle c_{m-5}\rangle\times\langle ac_{m-3}\rangle.
\]
It is easy to verify that $\Fix(x)=M\setminus\{1\}$ and thence
\[
\Fix(yxy)\cap\Fix(xyxyx)=\Fix(x)^y\cap\Fix(x)^{yx}=\emptyset.
\]
However, $a^2b(\prod_{i=1}^{(m-4)/2}c_{2i})(\prod_{i=0}^{(m-6)/4}c_{4i-1})$ is fixed by both $zxz$ and $xzxzx$. Thus,
\[
|\Fix(zxz)\cap\Fix(xzxzx)|>0=|\Fix(yxy)\cap\Fix(xyxyx)|,
\]
contrary to~\eqref{eqn-even}.

Finally assume that $m\equiv0\pmod{4}$. Then in the same vein as above we have $\Fix(zxz)\cap\Fix(xzxzx)=\emptyset$ while the element $b(\prod_{i=0}^{(m-4)/2}c_{2i})(\prod_{i=0}^{(m-4)/4}c_{4i-1})$ of $H^*$ is fixed by both $yxy$ and $xyxyx$. This causes
\[
|\Fix(yxy)\cap\Fix(xyxyx)|>0=|\Fix(zxz)\cap\Fix(xzxzx)|,
\]
contradicting~\eqref{eqn-even}.
\end{proof}

In the following lemma we prove that the full automorphism of $\Ga_m$ is isomorphic to $\A_{2^m}$. Some of the arguments here were used in the proof of~\cite[Theorem~1.2]{ZF2010}.

\begin{lemma}\label{lem15}
$\Aut(\Ga_m)\cong\A_{2^m}$.
\end{lemma}

\begin{proof}
Let $A=\Aut(\Ga_m)$ and $v$ be a vertex of $\Ga_m$. Then by Lemma~\ref{lem16}, $A$ has a nonnormal vertex-regular subgroup $G$ which is isomorphic to the alternating group $\A_{2^m-1}$. Further, $\Nor_A(G)=G$ by Lemma~\ref{lem17}. Note also that $\Ga_m$ is connected and cubic as Lemma~\ref{lem14} asserts. We derive from~\cite[Theorem~1.1]{XFWX2005} that $A$ is not transitive on the arc set of $\Ga_m$, and so $A_v$ is a $2$-group. Consequently, $|A|/|G|=|GA_v|/|G|=|A_v|/|G\cap A_v|=|A_v|$ is a power of $2$. Since every nontrivial $G$-conjugacy class has size greater than $3$, it follows from~\cite[Theorem~1.1]{FPW2002} that one of the following two cases occurs:
\begin{itemize}
\item[(i)] $\Soc(A)$ is a nonabelian simple group containing $G$ as a proper subgroup;
\item[(ii)] $A$ has a nontrivial normal subgroup $N$ such that $N$ is not transitive on the vertex set of $\Ga_m$ and $\Soc(A/N)$ is a nonabelian simple group containing $GN/N\cong G$.
\end{itemize}

First assume that case~(i) occurs. Then as $|\Soc(A)|/|G|$ is a power of $2$, we have $\Soc(A)=\A_{2^m}$ by~\cite[Theorem~1]{Guralnick1983}, and so $A\cong\A_{2^m}$ or $\Sy_{2^m}$. If $A\cong\Sy_{2^m}$, then $\Nor_A(G)\cong\Sy_{2^m-1}$, contrary to the conclusion that $\Nor_A(G)=G$. Therefore, $A\cong\A_{2^m}$.

Next assume that case~(ii) occurs. In this case, $N\cap G=1$ as $GN/N\cong G$. Hence $|N|=|N|/|N\cap G|=|NG|/|G|$ divides $|A|/|G|$. In particular, $N$ is a $2$-group. From the construction of $\Ga_m$ we know that $A$ has a subgroup $B$ that is isomorphic to $\A_{2^m}$ and contains $G$. Consider the action $\phi$ of $B$ on $N$ by conjugation. Since $B$ is a simple group, either $\ker(\phi)=1$ or $\ker(\phi)=B$. If $\ker(\phi)=B$, then $B$ centralizes $N$ and so $N\leqslant\Nor_A(G)=G$, contradicting the condition that $N\cap G=1$. Hence we have $\ker(\phi)=1$. Then $B\cong\A_{2^m}$ is isomorphic to an irreducible subgroup of $\Aut(N/\Phi(N))\cong\PSL_d(2)$ for some positive integer $d$ with $2^d\leqslant|N|$, where $\Phi(N)$ is the Frattini subgroup of $N$. It follows that $d\geqslant2^m-2$ according to~\cite[Proposition~5.3.7]{KL1990}. Hence $\nu_2(|N|)\geqslant2^m-2$, where $\nu_2$ is the $2$-adic valuation. Moreover, $N$ must be semiregular on the vertex set of $\Ga_m$, for otherwise the quotient graph of $\Ga_m$ with respect to $N$ would have valency $2$ and so could not admit $A/N$ as a group of automorphisms. Accordingly,
\[
\nu_2(|N|)\leqslant\nu_2(|\A_{2^m-1}|)=\sum_{i=1}^\infty\left\lfloor\frac{2^m-1}{2^i}\right\rfloor-1<\sum_{i=1}^\infty\frac{2^m-1}{2^i}-1=2^m-2.
\]
This contradicts the conclusion $\nu_2(|N|)\geqslant2^m-2$, not possible.
\end{proof}

\noindent\textsc{Acknowledgements.} This research was supported by the National Natural Science Foundation of China (Grant No.~11501011 and 11671030). The authors are very grateful to Prof.~Marston Conder for helpful discussion and the anonymous referees for their comments to improve the paper.


\begin{thebibliography}{}

\bibliographystyle{100}

\bibitem{Cameron1999}
P. J. Cameron,
{\it Permutation groups}, Cambridge University Press, Cambridge, 1999.

\bibitem{FLWX2002}
X. G. Fang, C. H. Li, J. Wang and M. Y. Xu,
On cubic Cayley graphs of finite simple groups,
{\it Discrete Math.}  244  (2002),  no. 1-3, 67--75.

\bibitem{FPW2002}
X. G. Fang, C. E. Praeger and J. Wang,
On the automorphism groups of Cayley graphs of finite simple groups,
{\it J. London Math. Soc. (2)}  66  (2002),  no. 3, 563--578.

\bibitem{FLX2008}
Y.-Q. Feng, Z.-P. Lu and M.-Y. Xu,
Automorphism groups of Cayley digraphs,
in {\it Application of Group Theory to Combinatorics},
edited by J. Koolen, J. H. Kwak and M. Y. Xu,
CRC Press, Taylor $\&$ Francis Group, London, 2008. pp. 13--25.

\bibitem{Godsil1981}
C. D. Godsil,
On the full automorphism group of a graph,
{\it Combinatorica} 1  (1981), no. 3, 243--256.

\bibitem{Godsil1983}
C. D. Godsil,
The automorphism group of some cubic Cayley graphs,
{\it European J. Combin.} 4  (1983),  no. 1, 25--32.

\bibitem{Guralnick1983}
R. M. Guralnick,
Subgroups of prime power index in a simple group,
{\it J. Algebra}  81  (1983), no. 2, 304--311.

\bibitem{KL1990}
P. B. Kleidman and M. W. Liebeck,
{\it The subgroup structure of the finite classical groups},
London Mathematical Society Lecture Note Series, 129, Cambridge University Press, Cambridge, 1990.

\bibitem{LLW2013}
J. Li, B. Lou and R. Wang,
On cubic nonsymmetric Cayley graphs,
{\it Open Journal of Discrete Mathematics} 3 (2013), no. 1, 39--42.

\bibitem{Mihailescu2004}
P. Mih\v{a}ilescu,
Primary cyclotomic units and a proof of Catalan's conjecture,
{\it J. Reine Angew. Math.} 572 (2004), 167--195.

\bibitem{PSV2013}
P. Poto\v{c}nik, P. Spiga and G. Verret,
Cubic vertex-transitive graphs on up to 1280 vertices,
{\it J. Symbolic Comput.} 50  (2013), 465--477.

\bibitem{Praeger1999}
C. E. Praeger,
Finite normal edge-transitive Cayley graphs,
{\it Bull. Austral. Math. Soc.} 60  (1999),  no. 2, 207--220.

\bibitem{WF2012}
X. Wang, Y.-Q. Feng,
Tetravalent half-edge-transitive graphs and non-normal Cayley graphs,
{\it J. Graph Theory}  70  (2012),  no. 2, 197--213.

\bibitem{Xu1998}
M.-Y. Xu,
Automorphism groups and isomorphisms of Cayley digraphs,
{\it Discrete Math.}  182  (1998),  no. 1-3, 309--319.

\bibitem{XFWX2005}
S. J. Xu, X. G. Fang, J. Wang and M. Y. Xu,
On cubic $s$-arc transitive Cayley graphs of finite simple groups,
{\it European J. Combin.}  26  (2005),  no. 1, 133--143.

\bibitem{XFWX2007}
S. J. Xu, X. G. Fang, J. Wang and M. Y. Xu,
$5$-arc transitive cubic Cayley graphs on finite simple groups,
{\it European J. Combin.}  28  (2007),  no. 3, 1023--1036.

\bibitem{ZF2010}
C. Zhang and X. G. Fang,
A note on the automorphism groups of cubic Cayley graphs of finite simple groups,
{\it Discrete Math.}  310  (2010),  no. 21, 3030--3032.

\end{thebibliography}
\end{document}